\documentclass[12pt,oneside]{amsart} 
\usepackage{geometry} 
\usepackage{color}

\usepackage{amssymb}
\usepackage{upgreek}

\theoremstyle{plain}
\newtheorem{theorem}{Theorem}[section]     
     
\newtheorem{lemma}[theorem]{Lemma}

\newtheorem{ex}[theorem]{Example}
\newtheorem{proposition}[theorem]{Proposition}
\newtheorem{define}[theorem]{Definition}

\theoremstyle{definition}
\newtheorem*{remark}{Remark}

\newenvironment{proofof}[1]{\noindent{\it Proof of
#1.}}{\hfill$\square$\\\mbox{}}

\newcommand{\ord}{\mathop{\rm ord}\nolimits}
\newcommand{\rank}{\mathop{\rm rank}\nolimits}
\newcommand{\md}{\mathop{\rm mod}\nolimits}
\newcommand{\supp}{\mathop{\rm supp}\nolimits}
\newcommand{\kn}{\mathop{\rm \uptheta}\nolimits}
\newcommand{\also}{\mathop{a}\nolimits}
\newcommand{\egyes}{\mathop{v}\nolimits}
\newcommand{\tort}{\mathop{\ell}\nolimits}
\newcommand{\linkomb}{\mathop{b}\nolimits}

\usepackage{url}

\begin{document}
\title[Separating Noether number]{The separating Noether number of abelian groups of rank two}
\author[Barna Schefler]{Barna Schefler}
\address{E\"otv\"os Lor\'and University, 
P\'azm\'any P\'eter s\'et\'any 1/C, 1117 Budapest, Hungary} 
\email{scheflerbarna@yahoo.com}
\thanks{Partially supported by the Hungarian National Research, Development and Innovation Office,  NKFIH K 138828.}

\maketitle

\begin{abstract} 
The exact value of the separating Noether number of an arbitrary finite abelian group of rank two is determined. This is done by a detailed study of the monoid of zero-sum sequences over the group. 
\end{abstract}

\section{Introduction} \label{sec:intro} 

\subsection*{Separating invariants}
The \emph{separating Noether number} 
$\beta_{sep}(G)$ 
of a finite group $G$ was introduced in \cite{kohls-kraft} as the minimal positive integer $d$ such that for any 
finite dimensional complex representation of $G$, 
the homogeneous polynomial $G$-invariants of degree at most $d$ form a separating set. The present paper fits into the program of determining $\beta_{sep}(G)$ for finite abelian groups $G$. 

Let $V$ be a finite dimensional vector space over $\mathbb{C}$ endowed with an action of $G$ via linear transformations. The algebra $\mathbb{C}[V]$ of polynomial functions on $V$ contains the subalgebra 
$\mathbb{C}[V]^G$ of $G$-invariants 
(i.e. polynomial functions constant along the $G$-orbits). 
By \cite[Theorem 3.12.1]{derksen-kemper} for any $v,w\in V$ for which $Gv\neq Gw$, there exists a homogeneous polynomial invariant $f\in \mathbb{C}[V]^G$ of degree at most $|G|$ with $f(v)\neq f(w)$. 
Denote by $\beta_{sep}(G,V)$ the minimal positive integer $d$ such that for any 
$v,w\in V$ with $Gv\neq Gw$, there exists a homogeneous $f\in \mathbb{C}[V]^G$ of degree at most $d$ satisfying $f(v)\neq f(w)$. Set  
\[\beta_{sep}(G):=\underset{V}{\sup}\{\beta_{sep}(G,V) \mid V \text{ is a finite dimensional }G\text{-module}\}. \]
So by \cite[Theorem 3.12.1]{derksen-kemper} we have $\beta_{sep}(G)\le |G|$. 

Denote by $C_n$ the cyclic group of order $n$. Our main result gives the exact value of the separating Noether number for any finite abelian group $C_{\tort n}\oplus C_n$ of rank two: 

\begin{theorem}\label{mainthm}
Let $\tort,n$ be positive integers and denote by $p$ the minimal prime divisor of $n$. Then we have: 
\[\beta_{sep}(C_{\tort n}\oplus C_{n})=\tort n+\frac{n}{p}\]
\end{theorem}

The definition of $\beta_{\mathrm{sep}}(G)$ is modeled on the \emph{Noether number} $\beta(G)$,  defined in \cite{schmid} as the maximal degree in a minimal homogeneous generating system of the algebras $\mathbb{C}[V]^G$, where $V$ ranges over all finite dimensional representations of $G$. The famous theorem of Noether asserts that  $\beta(G)\le |G|$. 
The exact value of the Noether number or the separating Noether number is known only for a few families of groups other than those covered by 
Theorem~\ref{mainthm} (see \cite{cziszter-domokos-szollosi}, \cite{domokos_abelian}, \cite{schefler_c_n^r}, \cite{3d+3+3}).

\subsection*{Davenport constant} By \cite[Corollary 2.6.]{domokos_abelian} 
(see Lemma~\ref{md2} of the present paper) Theorem~\ref{mainthm} translates to a question on zero-sum sequences over 
$C_{\tort n}\oplus C_{n}$. 

It is well known (see for example \cite{schmid}) that for a finite  abelian group $G$, $\beta(G)$  coincides with the \emph{Davenport constant} $\mathsf{D}(G)$ of $G$ (the maximal length of an irreducible zero-sum sequence over $G$). This observation highlights an interesting  connection between invariant theory and the theory of zero-sum sequences (see \cite{cziszter-domokos-geroldinger} for more information on this topic).  

A long-standing open question about the Davenport constant is a conjecture stating that for the direct sum $C_n^r$ of $r$ copies of the cyclic group of order $n$, the equality $D(C_n^r)=1+r(n-1)$ (or equivalently, 
$\beta(C_n^r)=1+r(n-1)$) holds (see \cite{gaogeroldinger1}, \cite{girard}). 
The analogous question on the separating Noether number was answered recently in 
\cite{schefler_c_n^r}, where 
we calculated the exact value of 
$\beta_{sep}(C_n^r)$, see Theorem \ref{ptlns} of the present paper for the result.\par 

On the other hand, the most common families of finite abelian groups for which the exact value of the Davenport constant (hence the Noether number) is known are the finite abelian groups of rank two and the finite abelian $p$-groups 
(\cite{rank2}, \cite{p-group}; see also \cite{d^*} for further information). 
It is therefore natural to aim at computing $\beta_{sep}(G)$ for rank two abelian groups as we do here. 
Interestingly, this seems to be technically more involved than the 
study of the separating Noether number of 
$C_n^r$ (whose Noether number is still unknown in general). 
The methods of this paper will be used in a subsequent paper \cite{schefler_rank3} 
to determine the separating Noether number for rank three finite abelian groups. 

\subsection*{Outline of the present paper}
The interpretation of the separating Noether number of a finite abelian group $G$ in terms of zero-sum sequences over $G$ is reviewed in Section~\ref{sec:separating-Davenport},  where we also recall some results from \cite{schefler_c_n^r} that are used here.  Key technical contributions of the present work are Proposition \ref{knodd} in Section \ref{sec:prep} and Lemma \ref{ord} in Section \ref{sec:lemma}. Their consequences are applied in  Section \ref{sec:proof} to  complete the proof of Theorem \ref{mainthm}.

\section{Preliminaries}\label{sec:separating-Davenport}
Let $g_1,...,g_k$ be distinct elements of the additively written finite abelian group $G$. Then 
\begin{center}
    $\mathcal{G}(g_1,...,g_k):=\{[m_1,...,m_k]\in\mathbb{Z}^k:\sum m_ig_i=0\in G\}$\\
\end{center}
is a subgroup of the additive group of $\mathbb{Z}^k$. Its submonoid 
\begin{center}
    $\mathcal{B}(g_1,...,g_k):=\mathbb{N}^k\cap \mathcal{G}(g_1,...,g_k)$
\end{center}   
is called \emph{block monoid}, see \cite[Definition 2.5.5]{ghk}. In the particular case when $g_1,...,g_k$ are all the elements of $G$, by slight abuse of notation we write $\mathcal{B}(G):=\mathcal{B}(g_1,...,g_k)$.\par 

The \emph{length} of an element $\mathsf{m}=[m_1,...,m_k]\in\mathcal{B}(g_1,...,g_k)$ is $|\mathsf{m}|=\sum_{i=1}^k m_i$. An element of the monoid $\mathcal{B}(g_1,...,g_k)$ that can not be written as the sum of two non-zero elements of the monoid is called an \emph{atom}. Let $\{g_1,...,g_k\}\subset\{g_0,g_1,...,g_k\}$, then each element $\mathsf{m_0}=[m_1,...,m_k]$ of the monoid $\mathcal{B}(g_1,...,g_k)$ can be identified with the element $\mathsf{m}=[0,m_1,...,m_k]$ of the monoid $\mathcal{B}(g_0,g_1,...,g_k)$. If $\mathsf{m_0}\in\mathcal{B}(g_1,...,g_k)$ is an atom, then $\mathsf{m}\in\mathcal{B}(g_0,g_1,...,g_k)$ is also an atom. By iterating this idea we get that each atom of a monoid $\mathcal{B}(g_1,...,g_k)$ can be identified with an atom of $\mathcal{B}(G)$. The maximal length of an atom in the monoid $\mathcal{B}(G)$ is denoted by $\mathsf{D}(G)$. Thus we get the definition of the Davenport constant of the abelian group $G$.\par

The support of an element is:
\begin{center}
    $\supp(\mathsf{m})=\{i\in\{1,...,k\}: m_i\neq 0\}$
\end{center}
The size of the support is $|\supp(\mathsf{m})|$. The $i^{th}$ standard basis vector of $\mathbb{Z}^k$ will be denoted by $\mathsf{e_i}$, hence $\ord(g_i)\mathsf{e_i}\in \mathcal{B}(g_1,...,g_k)$, where $\ord(g_i)$ the order of $g_i$ in $G$. This implies that  $\mathcal{B}(g_1,...,g_k)$ generates $\mathcal{G}(g_1,...,g_k)$ as a group.

\begin{remark}
One can think of $\mathcal{B}(G)$ as the monoid of zero-sum sequences over $G$. An element $\mathsf{m}=[m_1,...,m_k]\in \mathcal{B}(G)=\mathcal{B}(g_1,...,g_k)$ corresponds to the sequence over $G$ containing $g_i\in G$ with multiplicity $m_i$. For a survey on zero-sum sequences over abelian groups see  \cite{d^*}, \cite[Chapter 5]{ghk}, \cite[Chapter 10]{grynkiewicz}.
\end{remark}

The following definition was introduced in \cite{schefler_c_n^r}: 

\begin{define}\label{groupatomdef}
An element $\mathsf{m}\in\mathcal{B}(g_1,...,g_k)$ is called a \emph{group atom in 
$\mathcal{B}(g_1,...,g_k)$}  
if $\mathsf{m}$ can not be written as an \emph{integral} linear combination of elements 
of $\mathcal{B}(g_1,...,g_k)$ that have length strictly smaller than $|\mathsf{m}|$. 
\end{define}
\begin{remark}
(1) A group atom in $\mathcal{B}(g_1,...,g_k)$ is obviously an atom in 
$\mathcal{B}(g_1,...,g_k)$, but the converse is not true 
(see for example \cite[Example 4.4]{schefler_c_n^r}).

(2) It may happen that an element $\mathsf{m_0}=[m_1,...,m_k]$ is a group atom in some monoid $\mathcal{B}(g_1,...,g_k)$, however the element $\mathsf{m}=[0,m_1,...,m_k]$ is not a group atom in the larger monoid $\mathcal{B}(g_0,g_1,...,g_k)$ (see Example \ref{pelda} of the present paper).
\end{remark}

\begin{ex}\label{pelda}
Take the group $G=C_{12}\oplus C_4$, and consider the element $[8,4]\in\mathcal{B}((1,1),(1,2))$. For an element $\mathsf{m}=[m_1,m_2]\in\mathcal{B}((1,1),(1,2))$, we have $m_1+2m_2\equiv 0 \md 4$. So $m_1\equiv 0 \md 2$, hence by the condition $m_1+m_2\equiv 0 \md 12$, we get that $m_2\equiv 0 \md 2$ yielding that $m_1\equiv m_2\equiv 0 \md 4$. Hence the atoms in the monoid $\mathcal{B}((1,1),(1,2))$ are exactly: $[0,12],[4,8],[8,4],[12,0]$. This shows that $[8,4]$ is a group atom in the monoid $\mathcal{B}((1,1),(1,2))$, since it can not be written as a linear combination of elements of lower length, as such elements do not exist. However the equality $[0,8,4]=2[2,4,4]-[4,0,4]$ shows that $[0,8,4]$ is not a group atom in the monoid $\mathcal{B}((8,0),(1,1),(1,2))$. 
\end{ex}

The precise relation between the separating Noether number of an abelian group and the theory of zero-sum sequences  
was established in \cite{domokos_abelian} (see \cite{domokos_BLMS} 
for the case of not necessarily finite diagonalizable groups): 

\begin{lemma}\cite[Corollary 2.6.]{domokos_abelian}\label{md2}
The number $\beta_{sep}(G)$ is the maximal length of a group atom in $\mathcal{B}(g_1,...,g_k)$, where $\{g_1,\dots,g_k\}$ ranges over all subsets of size $k\leq \rank(G)+1$ of the abelian group $G$.
\end{lemma}

In this paper the following convention will be used: $G$ will stand for the abelian group $G=C_{n_1}\oplus C_{n_2}\oplus ... \oplus C_{n_r}$ where $2\le n_r\mid...\mid n_2\mid n_1$, and $r=2s$ or $2s-1$, depending on its parity, $g_1,\dots,g_k$ are fixed distinct elements of $G$, and $\mathcal{B}(g_1,\dots,g_k)$ the corresponding block monoid. Note that here $r$ is the \emph{rank} (i.e. the minimal number of generators), and $n_1=\exp(G)$ is the \emph{exponent} (i.e. the least common multiple of the orders  of the elements) of the group $G$.

Using Lemma \ref{md2}, the following general upper bound was given for $\beta_{sep}(G)$:

\begin{lemma}\cite[Theorem 3.10]{domokos_abelian}\label{md1}
$\beta_{sep}(G) \leq \sum_{i=1}^r(n_i-1)+1$ for all abelian groups $G$, with equality  if and only if $G$ is cyclic or $2=n_{s+1}=$ $\dots$ $= n_r$ where $r=2s-1$ or $r=2s$.
\end{lemma}

We shall use the following results from \cite{schefler_c_n^r}:

\begin{lemma}\cite[Lemma 3.1.]{schefler_c_n^r}\label{m*geqm}
Let $\mathsf{m}$ be a group atom in the monoid $\mathcal{B}(g_1,\dots,g_k)$ such that $|\mathsf{m}|> \max\{\ord(g_i)\mid i=1,\dots,k\}$. Then we have 
\begin{equation}\label{mot}
    2|\mathsf{m}|\le \sum_{i=1}^k\ord(g_i)
\end{equation}
\end{lemma}

\begin{lemma}\cite[Lemma 4.3.]{schefler_c_n^r}\label{kn+gcd}
Assume that $r=2s$ and $n_{s+1}=\dots=n_1$. Let $\mathsf{m}=[m_1,\dots,m_{2s+1}]\in\mathcal{B}(g_1,\dots,g_{2s+1})$, for elements $g_1,\dots,g_{2s+1}\in G$ satisfying the following conditions:
\begin{itemize}
    \item[(i)] $\ord(g_i)=n_1$ for $i\in\{1,\dots,2s\}$ 
    \item[(ii)] $|\mathsf{m}|>sn_1+\frac{\ord(g_{2s+1})}{p}$, where $p$ is the minimal prime divisor of $\ord(g_{2s+1})$.    
\end{itemize}
Then $\mathsf{m}$ is not a group atom in $\mathcal{B}(g_1,\dots,g_{2s+1})$.
\end{lemma}
\begin{lemma}\cite[Lemma 5.5.]{schefler_c_n^r}\label{bsgeqsumni}
Assume that $r=2s$ and let $p$ be a prime divisor of $n_r$. Then $\beta_{sep}(G)\geq n_1+\dots+n_s+\frac{n_{s+1}}{p}$.
\end{lemma}
\begin{theorem}\cite[Theorem 1.2.]{schefler_c_n^r}\label{ptlns}
 For positive integers $n\ge 2$ and $r$ denote by $C_n^r$ the direct sum $C_n\oplus\cdots\oplus C_n$ of $r$ copies of the cyclic group $C_n$ of order $n$, and let $p$ be the minimal prime divisor of $n$. 
 Then we have \textbf{}
 \[\beta_{sep}(C_n^r)=\begin{cases}
        ns,&\mbox{ if }r=2s-1 \mbox{ is odd} \\
        ns+\frac{n}{p},&\mbox{ if }r=2s \mbox{ is even}.\\
    \end{cases}\]
 \end{theorem}

\section{Preparations for the odd order case}\label{sec:prep}
\begin{lemma}\label{even}
Let $G$ be an abelian group of odd order, $g_1,\dots,g_k\in G$. Suppose that for an element $\mathsf{m}\in\mathcal{B}(g_1,\dots,g_k)$ there exists an integral linear combination $\lambda_1\mathsf{q}_1+...+\lambda_{t}\mathsf{q}_t+\varepsilon\mathsf{m=\linkomb}\in\mathcal{B}(g_1,\dots,g_k)$, $\mathsf{\linkomb}=[\linkomb_1,\dots,\linkomb_k]$ such that the  following conditions are fulfilled:
\begin{itemize}
    \item [(i)] $\lambda_j\in\mathbb{Z}$, $\mathsf{q}_j\in\mathcal{B}(g_1,\dots,g_k)$, $|\mathsf{q}_j|<|\mathsf{m}|$ for $j=1,...,t$, and $\varepsilon\in\{1,-1\}$
    \item [(ii)] $0\leq \linkomb_i=2\linkomb_i'$ is even for $i=1,\dots,k$ and $\linkomb_1+\dots+\linkomb_k<2|\mathsf{m}|$
\end{itemize}
Then $\mathsf{m}$ is not a group atom. 
\end{lemma}
\begin{proof}
Since the order of $G$ is odd, the equality $\sum_{i=1}^k 2\linkomb'_ig_i=0$ implies $\sum_{i=1}^k \linkomb'_ig_i=0$, hence $\mathsf{\linkomb'}=[\linkomb'_1,\dots,\linkomb'_k]\in\mathcal{B}(g_1,\dots,g_k)$. Here $|\mathsf{\linkomb'}|=\frac{\sum_{i=1}^k\linkomb_i}{2}<|\mathsf{m}|$, so the equality 
\[\varepsilon\mathsf{m}=2\mathsf{\linkomb'}-(\lambda_1\mathsf{q}_1+...+\lambda_{t}\mathsf{q}_t)\]
shows that $\mathsf{m}$ can be written as an integral linear combination of elements of length strictly lower than $|\mathsf{m}|$. Hence it is not a group atom.
\end{proof}

\begin{proposition}\label{knodd}
Let $G=C_{n_1}\oplus C_{n_2}\oplus ...\oplus C_{n_r}$ be an odd order abelian group of rank at least two, such that $1<\tort:=\frac{n_1}{n_2}$. Let $g_1,g_2,g_3\in G$ with $\ord(g_i)=\tort n$ for $i=1,2,3$, where $n=n_2$. Suppose that $\mathsf{m}=[m_1,m_2,m_3]\in\mathcal{B}(g_1,g_2,g_3)$ is an element such that $\tort n<|\mathsf{m}|$, and there is exactly one even number among $m_1,m_2$ and $m_3$. Then $\mathsf{m}$ is not a group atom.

\end{proposition}

\begin{proof}
Assume without loss of generality that $m_3$ is even, $m_1$ and $m_2$ are odd, and $m_1\leq m_2$. Note that $\tort n$, $n$ and $\tort$ are odd, since $|G|$ is odd. 

Since $\ord(g_1)=\ord(g_2)=\tort n$, both $ng_1$ and $-ng_2$ are generators of the order $\tort$ subgroup of the direct summand $C_{\tort n}$ of $G$. Hence there exists a positive integer $x\in\{0,...,\tort-1\}$ such that $-ng_2=xng_1$. Then $\mathsf{x}=[xn,n,0]\in\mathcal{B}(g_1,g_2,g_3)$, moreover $|\mathsf{x}|=n(x+1)\leq \tort n<|\mathsf{m}|$. Similarly, we have a positive integer $y\in\{0,...,\tort-1\}$, for which $\mathsf{y}=[n,yn,0]\in\mathcal{B}(g_1,g_2,g_3)$, and $|\mathsf{y}|\leq \tort n<|\mathsf{m}|$. Since $\ord(g_1)=\ord(g_2)=\tort n$, we have that $x\neq 0,y\neq 0$. Finally, by the same technique, we get a positive integer $z\in\{0,...,\tort-1\}$, for which $\mathsf{z}=[n,n,zn]\in\mathcal{B}(g_1,g_2,g_3)$. Note that if one of the equalities $x=1$; $y=1$; $z=0$ holds, then the other two also hold, yielding that $\mathsf{x=y=z}=[n,n,0]$.\par

If $z\leq \tort-2$, then $|\mathsf{z}|\leq 2n+n(\tort-2)=\tort n<|\mathsf{m}|$. If $z=\tort-1$, then $|\mathsf{z}|>\tort n$. Note that all the elements in $\mathcal{B}(g_1,g_2,g_3)$ of type $[\nu_1 n,\nu_2 n,\nu_3 n]$ correspond to elements $[\nu_1,\nu_2,\nu_3]\in\mathcal{B}(ng_1,ng_2,ng_3)$, where $ng_1,ng_2,ng_3$ are generators of the order $\tort$ subgroup of the direct summand $C_{\tort n}$ of $G$. Since $\beta(C_{\tort})=\tort$, atoms of $\mathcal{B}(C_{\tort})$ have length at most $\tort$, so atoms of $\mathcal{B}(g_1,g_2,g_3)$ of type $[\nu_1 n,\nu_2 n,\nu_3 n]$ have length at most $\tort n$. So if $z=\tort-1$, then $\mathsf{z}$ can be written as sum of elements of length at most $\tort n$.\par

Using $\mathsf{x,y,z}$ and $\mathsf{e_i}$ for $i=1,2,3$, we construct an integral linear combination $\lambda_1\mathsf{q_1}+...+\lambda_{t}\mathsf{q_t}+\varepsilon\mathsf{m=\linkomb}=[\linkomb_1,\linkomb_2,\linkomb_3]\in\mathcal{B}(g_1,g_2,g_3)$ fulfilling conditions (i) and (ii) of Lemma \ref{even}. Since $|\mathsf{e_i}|=\tort n<|\mathsf{m}|$, $|\mathsf{x}|<|\mathsf{m}|$, $|\mathsf{y}|<|\mathsf{m}|$ and either $|\mathsf{z}|<|\mathsf{m}|$, or $\mathsf{z}$ is the sum of elements of length at most $\tort n<|\mathsf{m}|$ (as described above), the first condition is automatically satisfied. We will distinguish cases depending on the parity of $x,y$ and $z$ and the size of $m_i$ compared to $n$.
\begin{enumerate}
    \item $x$ is odd:\\
    $\mathsf{\linkomb}:=\mathsf{m+x}$, so $[\linkomb_1,\linkomb_2,\linkomb_3]=[xn+m_1,n+m_2,m_3]$\\
    Here $0\leq\linkomb_1$ and $0\leq\linkomb_2$ are even, since $m_1,m_2,x$ and $n$ are odd, while $0\leq b_3=m_3$ is also even. Moreover, $\linkomb_1+\linkomb_2+\linkomb_3=|\mathsf{m}|+|\mathsf{x}|< 2|\mathsf{m}|$.
    \item $y$ is odd: similar to (1).
    \item $z$ is even and $z<\tort-1$:\\
    $\mathsf{\linkomb:=m+z}$, so $[\linkomb_1,\linkomb_2,\linkomb_3]=[m_1+n,m_2+n,m_3+zn]$\\ 
    $z\leq\tort -3$, so $|\mathsf{z}|=(z+2)n<\tort n<|\mathsf{m}|$. Hence $\linkomb_1+\linkomb_2+\linkomb_3=|\mathsf{m}|+|\mathsf{z}|< 2|\mathsf{m}|$.
\end{enumerate}
From now on, we assume that $x$ and $y$ are even, and $z$ is odd or $z=\tort-1$.
\begin{enumerate}\setcounter{enumi}{3}   
    \item $m_1<n\leq m_2$:\\
    $\mathsf{\linkomb}:=\tort n\mathsf{e_1}+\mathsf{m-x}$, so $[\linkomb_1,\linkomb_2,\linkomb_3]=[m_1+\tort n-xn,m_2-n,m_3]$\\
    $m_1+\tort n-xn+m_2-n+m_3=|\mathsf{m}|+\tort n-(x+1)n<2|\mathsf{m}|$.
    \item \begin{enumerate}
        \item $n\leq m_1\leq m_2$ and $z$ is odd:\\    
            $\mathsf{\linkomb:=m-z}+\tort n\mathsf{e_3}$, so $[\linkomb_1,\linkomb_2,\linkomb_3]=[m_1-n,m_2-n,m_3+\tort n-zn]$\\
            $m_1-n+m_2-n+m_3+\tort n-zn=|\mathsf{m}|+\tort n-(2+z)n<2|\mathsf{m}|$.
        \item $n\leq m_1\leq m_2$ and $z=\tort-1$:\\    
            $\mathsf{\linkomb:=m-z}+\tort n\mathsf{e_3}+\tort n\mathsf{e_3}$, so $[\linkomb_1,\linkomb_2,\linkomb_3]=[m_1-n,m_2-n,m_3+\tort n+n]$\\
            $m_1-n+m_2-n+m_3+\tort n+n=|\mathsf{m}|+\tort n-n<2|\mathsf{m}|$.
    \end{enumerate}
\end{enumerate}

\begin{enumerate}\setcounter{enumi}{5}
    \item \begin{enumerate}
        \item  $m_1\leq m_2<n$ and $z$ is odd:\\    
                    $\mathsf{\linkomb:=}\tort n\mathsf{e_3}+\mathsf{z-m}$, so $[\linkomb_1,\linkomb_2,\linkomb_3]=[n-m_1,n-m_2,\tort n+zn-m_3]$\\
                    $n-m_1+n-m_2+\tort n+zn-m_3=\tort n+(z+2)n-|\mathsf{m}|<2|\mathsf{m}|$.
                
        \item \begin{enumerate}
                \item $m_1\le m_2<n$, $z=\tort-1$, $m_3+n\leq \tort n$:\\  
                    $\mathsf{\linkomb:=z-m}$, so $[\linkomb_1,\linkomb_2,\linkomb_3]=[n-m_1,n-m_2,\tort n-n-m_3]$\\ 
                    $n-m_1+n-m_2+\tort n-n-m_3=\tort n+n-|\mathsf{m}|<2|\mathsf{m}|$.
                \item \begin{enumerate}
                        \item $m_1\le m_2<n$, $z=\tort-1$,              $m_3+n\geq \tort n$, $x=\tort-1$:\\
                           $\mathsf{b}:=\tort n\mathsf{e_1}+\tort n\mathsf{e_1}+\tort n\mathsf{e_2}+\tort n\mathsf{e_3}+\tort n\mathsf{e_3}-\mathsf{m-x-z}$, so\\
                           $[\linkomb_1,\linkomb_2,\linkomb_3]=[\tort n-m_1,\tort n-2n-m_2,\tort n+n-m_3]$\\
                           Here $\tort n-2n-m_2>0$, since $1<\tort$ is odd, and $m_2<n$.\\ 
                           $\tort n-m_1+\tort n-2n-m_2+\tort n+n-m_3=3\tort n-n-|\mathsf{m}|<2|\mathsf{m}|$.\\                        
                           In this case $\tort n\mathsf{e_1}+\tort n\mathsf{e_2-x-y}=[0,\tort n-(1+y)n,0]\in\mathcal{B}(g_1,g_2,g_3)$, hence $\tort n-(1+y)n$ is an integer multiple of $\ord(g_1)=\tort n$. This can happen only if $y=\tort-1$.   
                        \item $m_1\le m_2<n$, $z=\tort-1$,  
                           $m_3+n\geq \tort n$, $y=\tort-1$:\\ 
                           Similarly to (A). Moreover, $y=\tort-1$ implies $x=\tort-1$.
                        \item $m_1<n,m_2<n$, $z=\tort-1$, 
                           $m_3+n\geq \tort n$, $x\neq\tort-1$, $y\neq\tort-1$:\\ 
                           $\mathsf{b}:=\tort n\mathsf{e_1}+\tort n\mathsf{e_2+m-x-y-z}$, so\\
                           $[\linkomb_1,\linkomb_2,\linkomb_3]=[m_1+\tort n-(2+x)n,m_2+\tort n-(2+y)n,m_3-\tort n+n]$\\
                           Here $x\neq\tort-1$ is even, so $x\leq\tort-3$, implying $\tort-(2+x)\geq 1$. Hence $m_1+\tort n-(2+x)n>0$. Similarly, $m_1+\tort n-(2+y)n>0$.\\
                           $m_1+\tort n-(2+x)n+m_2+\tort n-(2+y)n+m_3-\tort n+n=|\mathsf{m}|+\tort n-(3+x+y)n<2|\mathsf{m}|$.
                    \end{enumerate}
                \end{enumerate}
        \end{enumerate} 
   \end{enumerate}
In all possible cases we presented an integral linear combination $\lambda_1\mathsf{q}_1+...+\lambda_{t}\mathsf{q}_t+\varepsilon\mathsf{m=\linkomb}$ fulfilling conditions (i) and (ii) of Lemma \ref{even}, hence $\mathsf{m}$ is not a group atom by Lemma \ref{even}.
\end{proof}

\section{A variant of Lemma \ref{m*geqm}}\label{sec:lemma}
The following notation and assumptions will be in effect all over Section \ref{sec:lemma}: 
\begin{itemize}

    \item $G$ is an abelian group 
    \item $g_1,g_2,g_3\in G$, where $\ord(g_1)\geq\ord(g_2)\geq\ord(g_3)$
    \item $\mathsf{m}=[m_1,m_2,m_3]\in\mathcal{B}(g_1,g_2,g_3)$  
    \item $H$ is a subgroup of $\langle g_1\rangle\cap\langle g_2\rangle\cap\langle g_3\rangle$ of order $\kn$
    \item $h_i:=\frac{\ord(g_i)}{\kn}g_i$, so $h_i$ is a generator of $H$ 
     for $i=1,2,3$
    \item $\also_i\in\{0,1,...,\kn-1\}$ with $\also_i\frac{\ord(g_i)}{\kn}\leq m_i<(\also_i+1)\frac{\ord(g_i)}{\kn}$
\end{itemize}
If $h_i=h_j$ for $i,j\in\{1,2,3\}$ and $i\neq j$, then we consider the block monoid $\mathcal{B}(h_1,h_2,h_3)$ in the obvious way. 
\begin{lemma}\label{vajon}
Take an element $\mathsf{u}=[u_1,u_2,u_3]\in\mathcal{B}(h_1,h_2,h_3)$. If $\kn>1$, then 
\begin{itemize}
    \item [(i)] $\mathsf{u}_{\kn}:=[u_1\frac{\ord(g_1)}{\kn},u_2\frac{\ord(g_2)}{\kn},u_3\frac{\ord(g_3)}{\kn}]\in\mathcal{B}(g_1,g_2,g_3)$.
    \item [(ii)] $|\mathsf{u}_{\kn}|\leq\frac{\ord(g_1)}{\kn}|\mathsf{u}|$
\end{itemize}
\end{lemma}

\begin{proof}
(i) Here $\sum_{i=1}^3u_ih_i=0\in H$ implies that $\sum_{i=1}^3u_i\frac{\ord(g_i)}{\kn}g_i=0\in G$. So $\mathsf{u}_{\kn}=[u_1\frac{\ord(g_1)}{\kn},u_2\frac{\ord(g_2)}{\kn},u_3\frac{\ord(g_3)}{\kn}]\in\mathcal{B}(g_1,g_2,g_3)$.\par

(ii) $|\mathsf{u}_{\kn}|=\sum_{i=1}^3u_i\frac{\ord(g_i)}{\kn}\leq\frac{\ord(g_1)}{\kn}\sum_{i=1}^3u_i=\frac{\ord(g_1)}{\kn}|\mathsf{u}|$.
\end{proof}

In the following lemma we improve inequality \eqref{mot}. Namely, we prove that under some specific conditions a positive term can be added to the left hand side and the inequality will still hold without modifying the right hand side. In this way we get a sharper upper bound for the length of the group atom in the given monoid. 

\begin{lemma}\label{ord}
Suppose that $\kn>1$ and $\mathsf{m}\in\mathcal{B}(g_1,g_2,g_3)$ is a group atom with $|\mathsf{m}|>\ord(g_1)$. Then the following hold:
\begin{itemize}
    \item [(i)] $\kn-1\geq\sum_{i=1}^3\also_i$
    \item [(ii)]  We have the equality $\sum_{i=1}^3\ord(g_i)\mathsf{e_i}=\mathsf{m}+\mathsf{\Tilde{m}}+\mathsf{m_{\egyes}}$ satisfying the following:
\end{itemize}   
\begin{itemize}
    \item $\mathsf{m_{\egyes}}=[\egyes_1\frac{\ord(g_1)}{\kn},\egyes_2\frac{\ord(g_2)}{\kn},\egyes_3\frac{\ord(g_3)}{\kn}]\in\mathcal{B}(g_1,g_2,g_3)$
    \item $\mathsf{m_{\egyes}}=\mathsf{\egyes}_{\kn}^1+...+\mathsf{\egyes}_{\kn}^N$
    \item for any $t\in\{1,...,N\}:$ $\mathsf{\egyes}_{\kn}^t\in\mathcal{B}(g_1,g_2,g_3)$ with $|\mathsf{\egyes}_{\kn}^t|\leq \ord(g_1)$.
    \item $\egyes_i\in\{0,1,...,\kn-1\}$ for $i=1,2,3$
    \item $\sum_{i=1}^3\egyes_i\geq \kn-1$
    \item $|\supp(\mathsf{m_{\egyes}})|\geq 2$
    \item $\mathsf{\Tilde{m}}\in\mathcal{B}(g_1,g_2,g_3)$ with $|\mathsf{\Tilde{m}}|\geq |\mathsf{m}|$
\end{itemize}
\begin{itemize}
\item [(iii)] $\kn\neq 2$, moreover, $2|\mathsf{m}|+\frac{\ord(g_{2})}{\kn}+(\kn-2)\frac{\ord(g_{3})}{\kn}\leq \sum_{i=1}^3\ord(g_i)$
\end{itemize}
\end{lemma}
\begin{proof}
(i) Suppose that $\sum_{i=1}^3\also_i\geq\kn$. Since $\beta(H)=|H|=\kn$, there exists an atom $\mathsf{u}=[u_1,u_2,u_3]\in\mathcal{B}(h_1,h_2,h_3)$, for which $u_i\leq \also_i$ for each $i=1,2,3$. Of course, $|\mathsf{u}|=\sum_{i=1}^3u_i\leq\beta(H)=\kn$. So for the coordinates of $\mathsf{u}_{\kn}=[u_1\frac{\ord(g_1)}{\kn},u_2\frac{\ord(g_1)}{\kn},u_3\frac{\ord(g_3)}{\kn}]\in\mathcal{B}(g_1,g_2,g_3)$ we have: $u_i\frac{\ord(g_i)}{\kn}\leq \also_i\frac{\ord(g_i)}{\kn}\leq m_i$. Hence $\mathsf{m-u_{\kn}}\in\mathcal{B}(g_1,g_2,g_3)$, and $\mathsf{m=u_{\kn}+(m-u_{\kn})}$, which implies that $\mathsf{m=u_{\kn}}$, since $\mathsf{m}$ is supposed to be a group atom in $\mathcal{B}(g_1,g_2,g_3)$. We get that $u_i\frac{\ord(g_i)}{\kn}=m_i$ for each $i=1,2,3$. If $|\mathsf{u}|=\kn$, then by Lemma \ref{vajon} (ii) we get that $|\mathsf{u}_{\kn}|\leq \ord(g_1)<|\mathsf{m}|$, contradiction. If $\kn-1\geq|\mathsf{u}|$, then $\kn-1\geq\sum_{i=1}^3u_i=\sum_{i=1}^3\also_i$, contradiction.\par
So we get that the indirect assumption was incorrect, hence $\kn-1\geq\sum_{i=1}^3\also_i$.\par 

(ii) Since $\kn>\sum_{i=1}^3\also_i$, we get that $\sum_{i=1}^3(\kn-\also_i-1)\geq \kn$, hence there exists an atom $\mathsf{\egyes}^1=[\egyes_1^1,\egyes_2^1,\egyes_3^1]\in\mathcal{B}(h_1,h_2,h_3)$, with $\egyes_i^1\leq \kn-\also_i-1$ for each $i=1,2,3$ (of course, $|\mathsf{\egyes}^1|\leq\beta(H)=\kn$). If $\sum_{i=1}^3(\kn-\also_i-1-\egyes_i^1)\geq \kn$, then there exists an atom $\mathsf{\egyes}^2=[\egyes_1^2,\egyes_2^2,\egyes_3^2]\in\mathcal{B}(h_1,h_2,h_3)$, such that $\egyes_i^1+\egyes_i^2\leq \kn-\also_i-1$ for each $i=1,2,3$. We continue this process until we have the sum $\mathsf{\egyes}^1+...+\mathsf{\egyes}^N=[\egyes_1,\egyes_2,\egyes_3]\in\mathcal{B}(h_1,h_2,h_3)$ satisfying the following:
\begin{itemize}
    \item [(a)] $\egyes_i\leq \kn-\also_i-1$ for $i=1,2,3$
    \item [(b)] for each atom $\mathsf{f}=[f_1,f_2,f_3]\in\mathcal{B}(h_1,h_2,h_3)$ there exists at least one index $i_0\in\{1,2,3\}$, for which $\egyes_{i_0}+f_{i_0}>\kn-\also_{i_0}-1$
\end{itemize}
(we may get different elements $[\egyes_1,\egyes_2,\egyes_3]$ for different choices of atoms $\mathsf{\egyes}^1,...,\mathsf{\egyes}^N$)\\
By Lemma \ref{vajon}, for each element $\mathsf{\egyes}^t\in\mathcal{B}(h_1,h_2,h_3)$ we get an element $\mathsf{\egyes}^t_{\kn}\in\mathcal{B}(g_1,g_2,g_3)$, such that $|\mathsf{\egyes}^t_{\kn}|\leq \ord(g_1)<|\mathsf{m}|$. Let us have the notation $\mathsf{m_{\egyes}}:=\mathsf{\egyes}^1_{\kn}+...+\mathsf{\egyes}^N_{\kn}\in\mathcal{B}(g_1,g_2,g_3)$ and $\mathsf{\Tilde{m}}:=\sum_{i=1}^3\ord(g_i)\mathsf{e_i}-\mathsf{m}-(\mathsf{\egyes}^1_{\kn}+...+\mathsf{\egyes}^N_{\kn})$. Thus we have the formula:
\begin{equation}\label{n+nh+nb}
\sum_{i=1}^3\ord(g_i)\mathsf{e_i}=\mathsf{m}+\mathsf{\Tilde{m}}+\mathsf{m_{\egyes}}
\end{equation}
By condition (a) we get that $\egyes_i\in\{0,...,\kn-1\}$. This also implies that for each coordinate of $\mathsf{m_{\egyes}}$ we have $\egyes_i\frac{\ord(g_i)}{\kn}<\ord(g_i)$, so $|\supp(\mathsf{m_{\egyes}})|>1$. By condition (b) we have that $\sum_{i=1}^3(\kn-\also_i-1-\egyes_i)< \kn$. Hence: 
\begin{center}
    $\sum_{i=1}^3\egyes_i\geq\sum_{i=1}^3(\kn-\also_i-1)-(\kn-1)\geq 3(\kn-1)-(\sum_{i=1}^3\also_i)-(\kn-1)\geq(\kn-1)$.
\end{center}
 $\mathsf{\Tilde{m}}\in\mathcal{B}(g_1,g_2,g_3)$, since its coordinates are nonnegative by condition (a). The formula 
\begin{center}
$\sum_{i=1}^3\ord(g_i)\mathsf{e_i}=\mathsf{m}+\mathsf{\Tilde{m}}+\mathsf{\egyes}^1_{\kn}+...+\mathsf{\egyes}_{\kn}^N$
\end{center}
consists of elements with length strictly smaller than $|\mathsf{m}|$, apart from $\mathsf{\Tilde{m}}$ and $\mathsf{m}$. Since $\mathsf{m}$ is supposed to be a group atom in $\mathcal{B}(g_1,g_2,g_3)$, this can happen only if $|\mathsf{\Tilde{m}}|\geq |\mathsf{m}|$.\par

(iii) Since $|\supp(\mathsf{m_{\egyes}})|\geq 2$, the least possible value of $|\mathsf{m_{\egyes}}|$ is reached, when $\egyes_1=0$ and $\egyes_2=\egyes_3=1$. However, if $\kn=2$, then this would give a contradiction:
\begin{center}
    $\sum_{i=1}^3\ord(g_i)\geq 2|\mathsf{m}|+\frac{\ord(g_2)}{2}+\frac{\ord(g_3)}{2}>2\ord(g_1)+\frac{\ord(g_2)}{2}+\frac{\ord(g_3)}{2}=\ord(g_1)+\frac{\ord(g_1)+\ord(g_2)}{2}+\frac{\ord(g_1)+\ord(g_3)}{2}\geq \sum_{i=1}^3\ord(g_i)$
\end{center}

Let us have $\kn>2$. Since $\ord(g_{1})\geq\ord(g_2)\geq\ord(g_3)$, $\sum_{i=1}^3\egyes_i\geq \kn-1$ and $|\supp(\mathsf{m_{\egyes}})|\geq 2$, the least possible value of $|\mathsf{m_{\egyes}}|$ is reached, when $\egyes_2=1$ and $\egyes_3=\kn-2$. By taking the length of the elements in \eqref{n+nh+nb} and using that $|\mathsf{\Tilde{m}|\geq|\mathsf{m}|}$, we get
\[\sum_{i=1}^3\ord(g_i)\geq 2|\mathsf{m}|+\frac{\ord(g_{2})}{\kn}+(\kn-2)\frac{\ord(g_{3})}{\kn}\qedhere\]
\end{proof}
\noindent
In addition to the notation set at the beginning of Section \ref{sec:lemma}, we assume the following:
\begin{itemize}
\item $G=C_{n_1}\oplus ...\oplus C_{n_r}$ is an abelian group with $r\geq 2$, such that $1<\tort:=\frac{n_1}{n_2}$
    \item $p$ is the minimal prime divisor of $n=n_2$
    \item $g_i=(g_{i,1},\sigma_i)\in C_{n_1}\oplus...\oplus C_{n_r}$, where $g_{i,1}\in C_{n_1}$ and $\sigma_i\in C_{n_2}\oplus...\oplus C_{n_r}$
    \item $\kn:=\gcd\left( \frac{\ord(g_1)}{\ord(\sigma_{1})},\frac{\ord(g_2)}{\ord(\sigma_{2})},\frac{\ord(g_3)}{\ord(\sigma_{3})}\right)$, 
\end{itemize}
Then $\langle g_1\rangle\cap\langle g_2\rangle\cap\langle g_3\rangle$ contains  a cyclic subgroup of order $\kn$ (contained also in the direct summand $C_{\tort n}$ of $G$).

\begin{proposition}\label{o>1}
Let  $\ord(g_1)=\ord(g_2)=\tort n$.
\begin{itemize}
    \item [(i)] if $\ord(g_3)$ does not divide $n$, then $\kn>1$
    \item [(ii)] $\frac{\ord(g_3)}{\kn}\leq n$
\end{itemize}

\end{proposition}
\begin{proof}
(i) Since $\kn\geq \gcd\left( \frac{\tort n}{n},\frac{\tort n}{n},\frac{\ord(g_3)}{\ord(\sigma_{3})}\right)$, it is enough to show that $\gcd\left(\tort,\frac{\ord(g_3)}{\ord(\sigma_{3})}\right)>1$. Suppose for contradiction, that $\gcd\left(\tort,\frac{\ord(g_3)}{\ord(\sigma_{3})}\right)=1$.\par

Then $\ord(\sigma_{3}) \frac{\ord(g_3)}{\ord(\sigma_{3})}=\ord(g_3)\mid \tort n=\tort\frac{n}{\ord(\sigma_{3})}\ord(\sigma_{3})$ hence $\frac{\ord(g_3)}{\ord(\sigma_{3})}\mid \tort\frac{n}{\ord(\sigma_{3})}$ so $\frac{\ord(g_3)}{\ord(\sigma_{3})}\mid \frac{n}{\ord(\sigma_{3})}$. This implies that $\ord(g_3)=\ord(\sigma_{3})\frac{\ord(g_3)}{\ord(\sigma_{3})}\mid \ord(\sigma_{3})\frac{n}{\ord(\sigma_{3})}=n$, contradicting the assumption on $\ord(g_3)$.\par

(ii) $\kn\geq\frac{\gcd(\tort n,\tort n,\ord(g_3))}{n}=\frac{\ord(g_3)}{n}$, so $\frac{\ord(g_3)}{\kn}\leq n$.
\end{proof}

\begin{proposition}\label{mgeqkn}
Let $\mathsf{m}\in\mathcal{B}(g_1,g_2,g_3)$ be a group atom with $|\mathsf{m}|>\ord(g_1)$.
\begin{itemize}
    \item [(i)] If $\ord(g_2)<\tort n$, then $|\mathsf{m}|\leq \tort n$.
    \item [(ii)] If $\ord(g_1)=\ord(g_2)=\tort n$ and $\ord(g_3)\nmid n$, then $\ord(g_3)=\tort n$ and we have the inequality 
    \begin{equation}\label{nodd}
\tort n+\frac{\ord(g_3)-\frac{\tort n}{2}}{\kn}\geq |\mathsf{m}|.
\end{equation}
\end{itemize}
\end{proposition}
\begin{proof}
(i) Lemma \ref{m*geqm} gives that $2|\mathsf{m}|\leq \ord(g_1)+\ord(g_2)+\ord(g_3)\leq \tort n+2\frac{\tort n}{2}$.\par

(ii) By choice of $\kn$, we have $\langle \frac{\tort n}{\kn}g_1\rangle=\langle \frac{\tort n}{\kn}g_2\rangle=\langle \frac{\ord(g_3)}{\kn}g_3\rangle\cong C_{\kn}$. By Proposition \ref{o>1} we get that $\kn>1$, hence we can use Lemma \ref{ord} (iii) for the group atom $\mathsf{m}$:
\[2\tort n+\ord(g_3)\geq2|\mathsf{m}|+\frac{\tort n}{\kn}+(\kn-2)\frac{\ord(g_3)}{\kn}\]
After rearranging, we obtain inequality \eqref{nodd}. Since $|\mathsf{m}|>\ord(g_1)=\tort n$, we get that $\frac{\ord(g_3)-\frac{\tort n}{2}}{\kn}>0$, so $\ord(g_3)=\tort n$.
\end{proof}

\begin{proposition}\label{keven}
If $\tort n$ is even, and $\ord(g_i)=\tort n$ for $i=1,2,3$, then for each group atom $\mathsf{m}\in\mathcal{B}(g_1,g_2,g_3)$ we have $|\mathsf{m}|\leq \tort n+\frac{n}{p}$. 
\end{proposition}

\begin{proof}
Suppose that $|\mathsf{m}|>\tort n$. If $\tort=2\ell'$ is even, then $\ord(g_i)=2\ell' n$, so $\langle \ell' ng_{1}\rangle=\langle \ell' ng_{2}\rangle=\langle \ell' ng_{3}\rangle\cong C_2$. However this contradicts Lemma \ref{ord} (iii), claiming that for a group atom $\mathsf{m}$ with $|\mathsf{m}|>\ord(g_1)$, the intersection $\langle g_1\rangle\cap\langle g_2\rangle\cap\langle g_3\rangle$ does not contain a  cyclic subgroup of order divisible by $2$.\par

So we can assume that $n$ is even, hence its minimal prime divisor is $p=2$. Proposition \ref{mgeqkn} (ii) implies that inequality \eqref{nodd} holds. So $\tort n+\frac{\tort n}{2\kn}\geq |\mathsf{m}|$, hence by Proposition \ref{o>1} (ii) we get that $|\mathsf{m}|\leq \tort n+\frac{n}{p}$. 
\end{proof}

\section{Proof of the main theroem}\label{sec:proof}

\begin{theorem}\label{ujt}
    Let $G=C_{n_1}\oplus ...\oplus C_{n_r}$ be an abelian group with $r\geq 2$, such that $1<\tort:=\frac{n_1}{n_2}$. Let $g_1,g_2,g_3\in G$ and $\mathsf{m}=[m_1,m_2,m_3]$ be an arbitrary group atom in $\mathcal{B}(g_1,g_2,g_3)$. Denote by $p$ the minimal prime divisor of $n=n_2$. Then 
    \[|\mathsf{m}|\leq\begin{cases}
        \tort n+\frac{n}{p} & \mbox{ if }r=2\\
        \tort n+\frac{n}{2} & \mbox{ if }r>2
    \end{cases}\]
\end{theorem}
\begin{proof}
Suppose that $|\mathsf{m}|>\tort n$. Then Proposition \ref{mgeqkn} (i) implies that $\ord(g_1)=\ord(g_2)=\tort n$, hence by Proposition \ref{mgeqkn} (ii) we get that either $\ord(g_1)=\ord(g_2)=\tort n$ with $\ord(g_3)\mid n$, or $\ord(g_i)=\tort n$ for $i=1,2,3$ holds.\par

Suppose that $\ord(g_1)=\ord(g_2)=\tort n$ with $\ord(g_3)\mid n$. Lemma \ref{m*geqm} implies that $2\tort n+\ord(g_3)\geq 2|\mathsf{m}|$, hence $|\mathsf{m}|\leq \tort n+\frac{n}{2}$. Moreover, if $r=2$ and $\ord(g_3)<n$ then $\ord(g_3)|n$ implies that $\ord(g_3)\leq\frac{n}{p}$, hence $|\mathsf{m}|\leq \tort n+\frac{\ord(g_3)}{2}<\tort n+\frac{n}{p}$. If $r=2$ and $\ord(g_3)=n$, then $|\mathsf{m}|\leq \tort n+\frac{n}{p}$ by Lemma \ref{kn+gcd}.
So we can turn to the case when $\ord(g_i)=\tort n$ for $i=1,2,3$.\par

Suppose that $\tort n$ is odd. If there was exactly one even number among $m_1,m_2$ and $m_3$, then $\mathsf{m}$ could not be a group atom by Proposition \ref{knodd}. If $m_1,m_2$ and $m_3$ were all even, then $\mathsf{m}$ itself would fulfill the conditions (i), (ii) of Lemma \ref{even}, while if $m_1,m_2$ and $m_3$ were all odd, then $\ord(g_1)\mathsf{e_1}+\ord(g_2)\mathsf{e_2}+\ord(g_3)\mathsf{e_3}-\mathsf{m}$ would fulfill the conditions (i), (ii) of Lemma \ref{even}. Finally, if there was exactly one odd number among $m_1,m_2$ and $m_3$ (assume that this one was $m_1$), then $\ord(g_1)\mathsf{e_1}+\mathsf{m}$ would fulfill the conditions (i), (ii) of Lemma \ref{even}. So Lemma \ref{even} would imply that $\mathsf{m}$ could not be a group atom.\par
Hence remains the case when $\tort n$ is even. Then $|\mathsf{m}|\leq \tort n+\frac{n}{p}$ by Proposition \ref{keven}, so we are done.
\end{proof}

\begin{proofof}{Theorem \ref{mainthm}}
If $\tort=1$, then $\beta_{sep}(C_{n}\oplus C_{n})=n+\frac{n}{p}$ by Theorem \ref{ptlns}.  So we can assume that $\tort>1$. Lemma \ref{bsgeqsumni} implies that $\beta_{sep}(C_{\tort n}\oplus C_{n})\geq \tort n+\frac{n}{p}$. By Lemma \ref{md2} we are searching for the maximal possible length of a group atom $\mathsf{m}$ in a monoid $\mathcal{B}(g_1,...,g_k)$, where $\{g_1,\dots,g_k\}$ ranges over all subsets of size $k\leq 3$ of the group $C_{\tort n}\oplus C_{n}$. If $|\supp(\mathsf{m})|=1$, then $|\mathsf{m}|\leq \tort n$. If $|\supp(\mathsf{m})|=2$, then $|\mathsf{m}|\leq \max\{\ord(g_1),\ord(g_2),\frac{\ord(g_1)+\ord(g_2)}{2}\}\leq\tort n$ by Lemma \ref{m*geqm}. If $|\supp(\mathsf{m})|=3$, then $|\mathsf{m}|\leq \tort n+\frac{n}{p}$ by Theorem \ref{ujt}.\par
So $\beta_{sep}(C_{\tort n}\oplus C_{n})\leq\tort n+\frac{n}{p}$, giving that $\beta_{sep}(C_{\tort n}\oplus C_{n})=\tort n+\frac{n}{p}$.
\end{proofof}

\end{document}